\documentclass[11pt, letter]{article}
\usepackage{amsmath}
\usepackage{amsfonts}
\usepackage{amssymb}
\usepackage{amsthm}
\usepackage{graphicx}
\usepackage{hyperref}
\usepackage{enumerate}

\usepackage{xcolor}

\newtheorem{Thm}{Theorem}[section]

\newtheorem{Cor}[Thm]{Corollary}
\newtheorem{Lemma}[Thm]{Lemma}
\newtheorem{Prop}[Thm]{Proposition}

\theoremstyle{definition}
\newtheorem{Def}[Thm]{Definition}

\theoremstyle{remark}
\newtheorem{Ex}[Thm]{Example}
\newtheorem{Rmk}[Thm]{Remark}
\numberwithin{equation}{section}

\begin{document}
	
	\title{Representation of solutions of the one-dimensional Dirac equation in terms of Neumann series of Bessel functions}
	\author{E. Roque, S.M. Torba \\{\small Departamento de Matem\'{a}ticas, Cinvestav, Unidad Quer\'{e}taro, }\\{\small Libramiento Norponiente \#2000, Fracc. Real de Juriquilla,
			Quer\'{e}taro, Qro., 76230 MEXICO.}\\{\small e-mail: earoque@math.cinvestav.edu.mx, storba@math.cinvestav.edu.mx}}
	
	\maketitle
	
	\begin{abstract}
		A representation of solutions of the one-dimensional Dirac equation is obtained. The solutions are represented as Neumann series of Bessel functions. The representations are shown to be uniformly convergent with respect to the spectral parameter. Explicit formulas for the coefficients are obtained via a system of recursive integrals. The result is based on the Fourier-Legendre series expansion of the transmutation kernel. An efficient numerical method for solving initial-value and spectral problems based on this approach is presented with a numerical example. The method can compute large sets of eigendata with non-deteriorating accuracy.
	\end{abstract}
	
	\section{Introduction}
	
	In this work, we consider the one-dimensional Dirac equation in canonical form
	\begin{equation}\label{eqn:dirac}
		B\frac{dY}{dx}+Q(x)Y=\lambda Y, \quad Y(x)=\begin{pmatrix}
			y_1(x)\\
			y_2(x)
		\end{pmatrix},
	\end{equation}
	on a finite interval \( x \in (0,b) \), where
	\begin{equation}
		B=\begin{pmatrix}
			0 & 1 \\ -1 & 0
		\end{pmatrix}, \quad
		Q(x)=\begin{pmatrix}
			p(x) & q(x) \\
			q(x) & -p(x)
		\end{pmatrix}, \quad p,q\in L^2([0,b],\mathbb{C}),
	\end{equation}
	and \(\lambda \) is an arbitrary complex constant known as the spectral parameter and \(Q\) is the matrix potential.\\

	The aim of this work is to obtain a representation of the matrix solution of the one-dimensional Dirac equation \eqref{eqn:dirac} in terms of Neumann series of Bessel functions (NSBF for short). Expansion of solutions of certain differential equations in terms of Bessel functions is not new. In \cite{Gersten}, Gersten expanded solutions for the radial Schrödinger equation in terms of spherical Bessel functions,  although the coefficients of the expansion depend solely on the spectral parameter, which differs from our approach. Later, in \cite{Chebli} the authors found  a representation of solutions of Sturm-Liouville equations in terms of Neumann series of Bessel functions. However, such representation does not possess some advantageous properties, such as uniformity with respect to the spectral parameter. Recently, Kravchenko, Navarro, and Torba found a new NSBF representation of solutions to the Schrödinger equation \cite{nsbf} that does possess uniformity properties. Since then, several works have been published that make use of this representation to solve a variety of direct and inverse spectral problems (see the book \cite{kravchenko-gb} and the references therein).  The present work extends the method to the system \eqref{eqn:dirac}.\\
	
	The one-dimensional Dirac equation has been an active research topic due to its importance both in mathematics and problems arising in mathematical physics, from the early works of Gasymov and Levitan \cite{Gasymov1,Gasymov2,Gasymov3} up to more recent works \cite{Hryniv, Shkalikov, HrynivManko} and numerous other publications. The one-dimensional Dirac equation appears in the inverse scattering method for the modified Korteweg-de Vries equation \cite{AKNS, AS}. It is also directly related to the Zakharov-Shabat system, which in turn appears in the inverse scattering method for non-linear Schrödinger equations, see \cite{HrynivManko} and the references therein. Furthermore, the one-dimensional Dirac equation can be related to the one-dimensional Schrödinger equation \cite[Section 8]{Nelson-analytic}, which is another equation of major importance. \\
	
	Despite the aforementioned importance of the Dirac system, there is still a need for fast, accurate, and efficient numerical methods to solve both initial and boundary value problems that pave the way towards solving other types of spectral problems numerically.\\
	
	Previous numerical methods, published in \cite{AnnabyTharwat,AnnabyTharwat2,AnnabyTharwat3,Tharwat}, share a common disadvantage due to the necessity to solve particular initial value problems for all sampling points, which results in high computational time. Later, in \cite{Nelson-spps} a spectral parameter power series representation for a pair of fundamental solutions of \eqref{eqn:dirac} was obtained by Gutiérrez-Jiménez and Torba, showing the possibility to compute dozens of eigenvalues with good accuracy and a relatively low computational cost. The same authors then found a better numerical method via an analytic approximation of the transmutation operator; see \cite{Nelson-analytic}. This work was certainly a big step in the right direction, as the authors showed that such a method can compute hundreds of eigenvalues with no-deteriorating accuracy and also possesses sought-for properties: efficient evaluation with respect to the spectral parameter, uniform error bound, and exponentially fast convergence. \\
	
	One might wonder why then do we need another numerical method to solve spectral problems for the Dirac equation. There are a few reasons for this. First, we obtain an exact representation for the solution and not an approximation. Moreover, the procedure found in the present work is considerably simpler to implement and also possesses the previously mentioned properties. Additionally, the NSBF representation for solutions of the Schrödinger equation found in \cite{nsbf} has proven to be a remarkable tool to solve not only direct spectral problems but also a wide class of inverse problems (see \cite{krav-sc, krav-icp,krav-qt, krav-cp}). Hence, the NSBF representation of the matrix solution of the Dirac system \eqref{eqn:Unsbf} brings up the possibility to tackle several spectral problems that, to the best of our knowledge, have not been solved numerically. \\
	
	The paper is organized as follows. In section \ref{sec:trans} we introduce the concept of a transmutation operator. A transmutation operator for the Dirac equation can be realized in the form of a Volterra integral operator \cite{levitan, Nelson-analytic}. In section \ref{sec:repr} we expand the kernel \(K(x,t)\) of the integral operator in terms of Legendre polynomials. The Fourier-Legendre expansion of the kernel is used to obtain the main result of this work, the representation \eqref{eqn:Unsbf}.  Then, we use \eqref{eqn:Unsbf} to obtain the NSBF representation of the matrix solution of the Zakharov-Shabat system \eqref{eqn:Znsbf}. We establish some results about the rate of convergence of the NSBF series in section \ref{sec:approx} in dependence of the smoothness of \(Q\). Finally, a numerical example is developed in section \ref{sec:numerical}. Appendix \ref{appendix:mapping} shows a connection between the coefficients of the NSBF representation \eqref{eqn:Unsbf} and the formal powers obtained in \cite{Nelson-spps}.
	
	\subsection*{Notations}
	Throughout this work we denote by \( \mathcal{M}_2 \) the algebra of \( 2\times 2 \) matrices with complex valued entries endowed with the operator norm induced by the Euclidean norm of the space \( \mathbb{C}^2 \). Moreover, the space \( L^2([a,b],\mathcal{M}_2) \) denotes the space of \( \mathcal{M}_2 \)-valued functions with the norm
	\begin{equation}
		\Vert A \Vert_2:= \left( \int_{a}^{b} \left\vert A(t) \right\vert^2 dt \right)^{1/2}.
	\end{equation}
	Additionally, \( W^{p,2}([a,b],\mathcal{M}_2)\) denotes the space of all matrix-valued functions \(A\) whose derivatives of order at most \(p-1\) are absolutely continuous functions and \(A^{(p)}\in L^2([a,b],\mathcal{M}_2) \). The space \( C^{p}([a,b],\mathcal{M}_2)\) is the space of \(p\)-times continuously differentiable matrix-valued functions and \( \Vert A \Vert_{C([a,b],\mathcal{M}_2)}=\max_{[a,b]}\vert A \vert.\)
	
\section{Transmutation operators}\label{sec:trans}
We introduce the definition of a tansmutation operator.
\begin{Def}[\cite{levitan}]\label{DefTransOpe} Let $E$ be a linear topological space and $E_{1}$ its linear subspace (not necessarily closed). Let $\mathcal A_{1},\mathcal A_{2}:E_{1}\rightarrow E$ be linear operators. A linear invertible operator $T$ defined on the whole $E$ such that $E_{1}$ is invariant under the action of $T$ is called a transmutation operator for the pair of operators $\mathcal A_{1}$ and $\mathcal A_{2}$ if it fulfills the following two conditions.
	\begin{enumerate}
		\item Both the operator $T$ and its inverse $T^{-1}$ are continuous in $E$;
		\item The following operator equality is valid  $$\mathcal A_{1}T=T\mathcal A_{2}$$ or which is the same $$\mathcal A_{1}=T\mathcal A_{2}T^{-1}.$$
	\end{enumerate}
\end{Def}

Let us denote by
\begin{equation}\label{eqn:diracop}
	\mathcal A_{Q}:=B\frac{d}{dx}+Q(x),
\end{equation}
a differential operator related to the system \eqref{eqn:dirac},
and by $\mathcal A_{0}$ the differential operator \eqref{eqn:diracop} having $Q$ as the null matrix-valued function. In order to introduce the transmutation operator we require an extension of \(Q \) to the symmetric interval \( [-b,b] \). However, none of our results depend on the values of \( Q \) outside of the interval \( [0,b]  \). Thus, we consider such extension when necessary.
\begin{Thm}\cite[Thm 2.2]{Nelson-analytic}\label{thm:transmutation}
	Suppose that \(Q \) is a continuously differentiable matrix-valued function on \( [-b,b] \). Then a transmutation operator \(T: E \to E, \, E:=C([-b,b], \mathbb{C}^2) \), relating the operators \( A_0 \) and \( A_Q \) for all \( Y \in E_1:=C^1([-b,b],\mathbb{C}^2 ) \), can be realized in the form of a Volterra integral operator
	\begin{equation}\label{eqn:top}
		TY(x) = Y(x) + \int_{-x}^{x} K(x,t)Y(t)dt,
	\end{equation}
	where \( K(x,t) \) is a \( 2\times 2 \) matrix-valued function satisfying the partial differential equation
	\begin{equation}\label{eqn:pdek}
		BK_x(x,t)+K_t(x,t)B=-Q(x)K(x,t)
	\end{equation}
	on the domain \( \Omega:=\{ (x,t): 0\leq \left\vert x \right\vert \leq b, \left\vert t \right\vert \leq x \}  \) with the Goursat conditions
	\begin{align}
		BK(x,x)-K(x,x)B=-Q(x), \label{eqn:goursatx}\\
		BK(x,-x)+K(x,-x)B=0. \label{eqn:goursatmx}
	\end{align}
	Conversely, if \( K(x,t) \) is a solution of the Goursat problem \eqref{eqn:pdek}, \eqref{eqn:goursatx},\eqref{eqn:goursatmx}, then the operator \(T \) determined by formula \eqref{eqn:top} is a transmutation operator for the pair of operators \( A_Q \) and \( A_0\).
\end{Thm}
\begin{Rmk}
	The existence of transmutation operators for the Dirac differential operator \eqref{eqn:diracop} in \( L^p\) spaces has also been establish in \cite{Hryniv}.
\end{Rmk}
The domain \( \Omega \) is the union of the sets
\begin{equation}
	\Omega^+:=\{ (x,t): 0\leq  x \leq b, \left\vert t \right\vert \leq x \}, \quad \Omega^-:=\{ (x,t): -b\leq  x \leq 0, \left\vert t \right\vert  \leq \vert x \vert \},
\end{equation}
and the Goursat problem can be solved independently on each. The solution \(K(x,t)\) on \(\Omega^+\) depends solely on the values of \(Q\) in the interval \([0,b]\). Moreover, for \( 0<x\leq b\), the integral formula \eqref{eqn:top} only requires the knowledge of the integral kernel \( K \) on the set \( \Omega^+ \). Hence, it only depends on the values of \(Q\) in the interval \([0,b]\). \\

Because of the transmutation property, a transmutation operator for the pair \( \mathcal A_0 \) and \( \mathcal A_Q \) sends solutions to the equation \( \mathcal A_0Y = \lambda Y \) to solutions of the equation \( \mathcal A_Q V = \lambda V \). Indeed,
\[
\mathcal{A}_Q(TY)= (\mathcal{A}_QT)Y=(T \mathcal{A}_0)Y=T(\mathcal{A}_0Y)=T (\lambda Y)= \lambda (T Y).
\]
\begin{Rmk}
	Notice that the transmutation operator, given by eq. \eqref{eqn:top}, acts on vector-valued functions \( Y \in C([-b,b],\mathbb{C}^2 ) \). Nevertheless, it is well-defined to evaluate the operator on \(2\times 2 \) matrix-valued functions, as the operator would act column by column in such a case. It is rather common to work with a pair of fundamental solutions to the Dirac equation \eqref{eqn:dirac}, namely, the solutions \( C(\lambda, x) \) and \( S(\lambda, x) \) that satisfy the initial conditions
	\[
	C(\lambda,0)=\begin{pmatrix} 1 \\ 0 \end{pmatrix}, \quad S(\lambda,0)=\begin{pmatrix} 0 \\ 1 \end{pmatrix}.
	\]
	In this article, we work directly with matrix-valued functions and the fundamental matrix solution of the Dirac equation, and the NSBF representation for \( C(\lambda, x) \) and \( S(\lambda, x) \) are given by the first and second column of the matrix-valued NSBF representation.
\end{Rmk}
If \( Q \equiv 0 \), the matrix solution \(U_0(\lambda, x) \) to \eqref{eqn:dirac} that satisfies the initial condition \( U_0(\lambda,0)=I \) is given by
\begin{equation}\label{eqn:u0}
	U_0(\lambda, x)=
	\begin{pmatrix}
		\cos \lambda x & -\sin \lambda x \\
		\sin \lambda x & \cos \lambda x
	\end{pmatrix}.
\end{equation}
Thus, the knowledge of the nucleus \( K  \) in the domain \( \Omega^+ \) is sufficient to find the matrix solution \( U(\lambda,x) \) to the Dirac equation \eqref{eqn:dirac} with initial condition \( U(\lambda,0)=I \) via the relation
\begin{equation}\label{eqn:transp}
	U(\lambda, x)=U_0(\lambda, x)+\int_{-x}^{x} K(x,t) U_0(\lambda, t)dt.
\end{equation}
At first glance, it might appear that the use of transmutation operators leads us to a more difficult task, namely, the solution of a Goursat problem. Solving a Goursat problem with an analytical procedure that is feasible for numerical implementation can be challenging. However, we overcome this issue by using a different approach, which is the Fourier-Legendre series expansion of the integral nucleus of the transmutation operator. We then find a practical method for computing its coefficients.

\section{The NSBF representation}\label{sec:repr}

Throughout this section we  assume that \( Q\in C^1([0,b], \mathcal{M}_2) \). Due to the fact that \(K\in C^1(\Omega^+, \mathcal{M}_2) \) (see proposition \ref{prop:smoothness}) it admits a uniformly convergent Fourier-Legendre series. Thus, the transmutation kernel matrix can be expanded in terms of the Legendre polynomials \(P_n \) as
\begin{equation}\label{eqn:legendre}
	K(x,t)=\sum_{n=0}^{\infty} \frac{1}{x}K_n(x) P_n(t/x)
\end{equation}
where
\begin{equation}\label{eqn:kappan}
	K(x,t)=\begin{pmatrix}
		\kappa_{11}(x,t) & \kappa_{12}(x,t) \\
		\kappa_{21}(x,t) & \kappa_{22}(x,t)
	\end{pmatrix},
	K_n(x)=\begin{pmatrix}
		\kappa^{(n)}_{11}(x) & \kappa^{(n)}_{12}(x) \\
		\kappa^{(n)}_{21}(x) & \kappa^{(n)}_{22}(x)
	\end{pmatrix}.
\end{equation}
\begin{Rmk}
	Due to the orthogonality of Legendre polynomials, multiplying \eqref{eqn:legendre} by \(P_n(\frac{t}{x}) \) and integrating we see that
	\begin{equation}\label{eqn:Kncoeff}
		\int_{-x}^{x}K(x,t)P_n \left( \frac{t}{x} \right)dt=\sum_{j=0}^{\infty}\frac{K_j(x)}{x}\int_{-x}^{x}P_j \left( \frac{t}{x} \right)
		P_n \left( \frac{t}{x} \right) dt = \frac{2}{2n+1}K_n(x).
	\end{equation}
\end{Rmk}
\begin{Rmk}\label{rmk:k0}
	Notice that for \( \lambda =0 \), equation \eqref{eqn:transp} together with \eqref{eqn:Kncoeff} show that
	\begin{equation}
		U(0,x)=I+\int_{-x}^{x}K(x,t) \cdot I dt=I+2K_0(x)
	\end{equation}
	In other words, computing the coefficient \( K_0 \) is equivalent to computing the fundamental solution to \( B\frac{dY}{dx} +Q(x)Y=0 \). This fact will be useful later on.
\end{Rmk}
Due to eq. \eqref{eqn:transp} and the fact that the series \eqref{eqn:legendre} is uniformly convergent, we have that
\begin{align*}
	U(\lambda, x) &= U_0(\lambda, x)+\int_{-x}^{x}K(x,t)U_0(\lambda,t)dt \\
	&= U_0(\lambda, x) + \sum_{n=0}^{\infty} \frac{1}{x} K_n(x) \int_{-x}^{x} U_0(\lambda,t) P_n \left( \frac{t}{x} \right) dt.
\end{align*}
On the other hand,
\begin{equation}
	\frac{1}{x}\int_{-x}^{x} U_0(\lambda, t) P_{2n}\left(\frac{t}{x} \right) dt = \frac{1}{x} \int_{-x}^{x} \begin{pmatrix}
		\cos \lambda x & -\sin \lambda x \\
		\sin \lambda x & \cos \lambda x
	\end{pmatrix} P_{2n}\left(\frac{t}{x} \right) dt = 2(-1)^n j_{2n}(\lambda x) I,
\end{equation}
where \(j_k\) stands for the spherical Bessel function of order \(k\). Similarly,
\begin{equation}
	\frac{1}{x}\int_{-x}^{x} U_0(t) P_{2n+1}\left(\frac{t}{x} \right) dt
	= -2(-1)^n j_{2n+1}(\lambda x) B.
\end{equation}
Therefore,
\begin{equation}\label{eqn:Unsbf}
	\begin{split}
		U(\lambda, x)  &= U_0(\lambda, x) + \sum_{n=0}^{\infty} \frac{1}{x} K_n(x) \int_{-x}^{x} U_0(\lambda, t) P_n \left( \frac{t}{x} \right) dt \\
		& = U_0(\lambda, x) + 2 \sum_{n=0}^{\infty}(-1)^n K_{2n}(x)j_{2n}(\lambda x)-2\sum_{n=0}^{\infty}(-1)^n K_{2n+1}(x)B j_{2n+1}(\lambda x).
	\end{split}
\end{equation}
The series \eqref{eqn:Unsbf} is known as a Neumann series of Bessel functions. In the next subsection we find a system of recursive differential equations for the coefficients \(K_n(x)\).
\subsection{A recursive system of equations}
First, we begin to formally derive our results, and in the next section justify our procedure. Differentiating the series term by term with respect to \(x \) we get
\begin{align*}
	U'(\lambda,x)&=U_0'(\lambda, x) + 2\sum_{n=0}^{\infty} (-1)^n \left[ K_{2n}'(x) j_{2n}(\lambda x) + \lambda K_{2n}(x) j_{2n}'(\lambda x) \right] \\
	& \,-2 \sum_{n=0}^{\infty} (-1)^n \left[ K_{2n+1}'(x)Bj_{2n+1}(\lambda x) + \lambda K_{2n+1}(x) B j_{2n+1}'(\lambda x) \right]
\end{align*}
Using the formulas (see \cite[f. 10.1.21, 10.1.22]{Abra})
\begin{align*}
	j_k'(z)&=-j_{k+1}(z) + \frac{k}{z}j_k(z), \quad k=0,1,\ldots, \\
	j_k'(z)&=j_{k-1}(z)-\frac{k+1}{z} j_k(z), \quad k=1,2,\ldots,
\end{align*}
we obtain that
\begin{align*}
	\lambda j_{2n}'(\lambda x) &= -\lambda j_{2n+1}(\lambda x) + \frac{2n}{x}j_{2n}(\lambda x) \\
	\lambda j_{2n+1}'(\lambda x) &= \lambda j_{2n}(\lambda x) - \frac{2n+2}{x}j_{2n+1}(\lambda x).
\end{align*}
We can simplify the derivative of \(U \) to
\begin{equation}\label{eqn:dUnsbf}
	\begin{split}
		U'(\lambda, x) &=U_0'(\lambda, x) +2 \sum_{n=0}^{\infty}(-1)^n\left[ K_{2n}'(x)+\frac{2n}{x}K_{2n}(x)-\lambda K_{2n+1}(x)B \right]j_{2n}(\lambda x) \\
		& \, -2\sum_{n=0}^{\infty} (-1)^n\left[ K_{2n+1}'(x)B -\frac{2n+2}{x} K_{2n+1}B+\lambda K_{2n}(x)\right]j_{2n+1}(\lambda x).
	\end{split}
\end{equation}
Formally substituting into eq. \eqref{eqn:dirac}, we get that \( BU'(\lambda, x)+Q(x)U(\lambda, x)-\lambda U(\lambda, x)=0 \) if and only if
\begin{equation}\label{eqn:nsbfv1}
	\begin{split}
		0&=Q(x)U_0(\lambda, x)+2(BK_0'(x)+Q(x)K_0(x) )j_0(\lambda x)  \\
		&\,+2\sum_{n=1}^{\infty} (-1)^n \left[BK_{2n}'(x) + \frac{2n}{x}BK_{2n}(x)+Q(x)K_{2n}(x) \right] j_{2n}(\lambda x)\\ 
		&\,-2\lambda\sum_{n=0}^{\infty} (-1)^n \left[BK_{2n+1}(x)B+K_{2n}(x) \right] j_{2n}(\lambda x)  \\
		&\,-2\sum_{n=0}^{\infty} (-1)^n \left[BK_{2n+1}'(x)B-\frac{2n+2}{x}BK_{2n+1}(x)B+Q(x)K_{2n+1}(x)B \right]j_{2n+1}(\lambda x) \\
		&\,-2\lambda \sum_{n=0}^{\infty} (-1)^n \left[BK_{2n}(x)-K_{n+1}(x)B \right] j_{2n+1}(\lambda x).
	\end{split}
\end{equation}
On the other hand,
\begin{equation}
	Q(x)U_0(\lambda, x)=Q(x)(\cos( \lambda x) I-\sin( \lambda x) B ) = \cos (\lambda x) Q(x) - \lambda x j_0(\lambda x) Q(x)B.
\end{equation}
Due to remark \ref{rmk:k0}, we have that
\begin{equation}
	K_0(x)=\frac{1}{2} \left( U(0,x)-I \right).
\end{equation}
Thus,
\begin{equation}
	2\left(BK_0'(x)+Q(x)K_0(x) \right)j_0(\lambda x)=-Q(x)j_0(\lambda x).
\end{equation}
Therefore,
\begin{equation}\label{eqn:qu0j}
	\begin{split}
		Q(x)U_0(\lambda, x)+2\left( BK_0'(x)+Q(x)K_0(x)\right)j_0(\lambda x)&=\left(\cos (\lambda x) - j_0(\lambda x) \right)Q(x)-
		\lambda x j_0(\lambda x)Q(x)B \\
		&=-\lambda x j_1( \lambda x )Q(x)-\lambda x j_0(\lambda x) Q(x)B.
	\end{split}
\end{equation}
Applying the recursive formula for spherical Bessel functions \cite[f. 10.1.19]{Abra}
\begin{equation}
	j_n(\lambda x)=\frac{\lambda x}{2n+1}(j_{n-1}(\lambda x)+j_{n+1}(\lambda x)),
\end{equation}
to the first and third series in eq. \eqref{eqn:nsbfv1}, respectively, in conjuction with \eqref{eqn:qu0j}, after division by \(-2\lambda x \), we can rewrite eq. \eqref{eqn:nsbfv1} as a series of the form
\begin{equation}
	\sum_{n=0}^{\infty}\alpha^{(1)}_n(x) j_{2n}(\lambda x)+\sum_{n=0}^{\infty} \alpha^{(2)}_n(x)j_{2n+1}(\lambda x)=0,
\end{equation}
where
\begin{equation}\label{eqn:alphaev}
	\begin{split}
		\alpha^{(1)}_n(x) = (-1)^n &\left\{  \frac{1}{4n+3} \left[ BK_{2n+1}'(x)-\frac{2n+2}{x}BK_{2n+1}(x)+QK_{2n+1}(x) \right] \right. \\
		&\, -\frac{1}{4n-1} \left[ BK_{2n-1}'(x)+\frac{2n}{x}BK_{2n-1}(x)+Q(x)K_{2n-1}(x) \right] \\
		&\, \left. + \frac{BK_{2n+1}(x)}{x}-\frac{K_{2n}(x)B}{x} \right\}B,
	\end{split}
\end{equation}
and
\begin{equation}\label{eqn:alphaodd}
	\begin{split}
		\alpha^{(2)}_n(x) = (-1)^n &\left\{  \frac{1}{4n+5} \left[BK_{2n+2}'(x)+\frac{2n+2}{x}BK_{2n+2}(x)+Q(x)K_{2n+2}(x) \right]  \right. \\
		&\, \frac{1}{4n+1} \left[BK_{2n}'(x)+\frac{2n}{x}BK_{2n}(x)+QK_{2n}(x) \right] \\
		&\, \left. + \frac{BK_{2n}(x)}{x}-\frac{K_{2n+1}(x)B}{x} \right\}.
	\end{split}
\end{equation}
We have introduced \( K_{-1}(x):=-K_0(x) \) to extend the formula \eqref{eqn:alphaev} to \( n=0 \). Due to the fact that the spherical Bessel functions are orthogonal in the interval \( ( -\infty, \infty ) \) ( cf. \cite[Formulas 10.1.2, 11.4.6]{Abra}) we conclude that
\begin{equation}\label{eqn:alphas0}
	\alpha_n^{(1)}=\alpha_n^{(2)}=0, \qquad n=0,1,\ldots,
\end{equation}
which gives us a coupled system of recursive differential equations for the coefficients \( K_n \), namely,
\begin{equation}\label{eqn:Kodd-rec}
	\begin{split}
		0=&  \frac{1}{4n+3} \left[ BK_{2n+1}'(x)-\frac{2n+2}{x}BK_{2n+1}(x)+QK_{2n+1}(x) \right] \\
		&\, -\frac{1}{4n-1} \left[ BK_{2n-1}'(x)+\frac{2n}{x}BK_{2n-1}(x)+Q(x)K_{2n-1}(x) \right] \\
		&\,  + \frac{BK_{2n+1}(x)}{x}-\frac{K_{2n}(x)B}{x},
	\end{split}
\end{equation}
and
\begin{equation}\label{eqn:Keven-rec}
	\begin{split}
		0 =  &  \frac{1}{4n+5} \left[BK_{2n+2}'(x)+\frac{2n+2}{x}BK_{2n+2}(x)+Q(x)K_{2n+2}(x) \right]   \\
		&\, \frac{1}{4n+1} \left[BK_{2n}'(x)+\frac{2n}{x}BK_{2n}(x)+QK_{2n}(x) \right] \\
		&\,  + \frac{BK_{2n}(x)}{x}-\frac{K_{2n+1}(x)B}{x} .
	\end{split}
\end{equation}
After applying the change of variables \( \theta_n=x^nK_n \) in eq. \eqref{eqn:Kodd-rec} and eq. \eqref{eqn:Keven-rec}, we get a single system of recursive equations, given by
\begin{equation}\label{eqn:diff-theta}
	\mathcal{A}_Q [\theta_n] = \frac{2n+1}{2n-3} \left[ x^2\mathcal{A}_Q[\theta_{n-2}]-(2n-3)xB\theta_{n-2}+(2n-3)\theta_{n-1}B  \right],
\end{equation}
which holds for all \( n \geq 1 \). Observe that 
\begin{equation}\label{eqn:thetan0}
	\theta_n(0)=0, 
\end{equation}
for \(n\geq 1\) since \( K_n(x) \) are continuous functions for \(x>0\) and bounded at \(x=0\). This follows from \eqref{eqn:Kncoeff}.
\begin{Rmk}
	Since \( Q\in C^1([0,b],\mathcal{M}_2) \), it can be shown using L'Hopital's rule that
	\[
	\theta_{-1}(x)=\frac{1}{x}K_{-1}(x)=-\frac{1}{x}K_0(x)\in C^1([0,b],\mathcal{M}_2).
	\]
\end{Rmk}

\subsection{Solving the system of recursive equations}
This subsection is devoted to the solution of the system \eqref{eqn:diff-theta}. We seek a matrix solution to this recursive system of differential equations that satisfies the initial condition \eqref{eqn:thetan0} for \( n \geq 1 \).  \\

It is well-known that the solution of an initial value problem of a non-homogeneous system of linear equations  can be written in terms of the fundamental matrix solution of the homogeneous system. We briefly recall here some of these facts for readability purposes, more details can be found in plenty of books, for instance \cite[Ex. 1.6.9 \& Section 1.8]{Miklavcic}. Let \( A \in  L^1([a,b],\mathcal{M}_d) \) be a \( d\times d \) matrix-valued function with integrable entries. Let \( X_0(t) \) be the solution of the initial-value problem
\begin{equation}\label{eqn:X0}
	X_0'(t)+A(t)X_0(t)=0, \quad X_0(a)=I,\, t\in [a,b].
\end{equation}
The function \(X_0(t)\) is called the fundamental matrix solution. Additionally, let \( Y_0(t) \) be the solution of the initial-value problem
\begin{equation}\label{eqn:adjoint}
	Y_0'(t)-A^*(t)Y_0(t)=0, \quad Y_0(a)=I, \, t\in [a,b],
\end{equation}
where \( A^* \) is the conjugate transpose of \( A \). It can be shown that \( X_0(t) \) is an invertible matrix, and its inverse is given by
\begin{equation}\label{eqn:X0inv}
	X_0^{-1}(t) = Y_0^*(t), \quad t\in [a,b].
\end{equation}
Then, the unique solution to the following initial-value problem
\begin{equation}
	X'(t)+A(t)X(t)=F(t), \quad  F \in L^1([a,b],\mathcal{M}_d), \quad   X(a)=X_a \in \mathcal{M}_d,
\end{equation}
is given by
\begin{equation}
	X(t)=X_0(t)X_a+X_0(t)\int_{a}^{t}Y_0^*(s)F(s)ds=X_0(t)X_a+X_0(t)\int_{a}^{t}X_0^{-1}(s)F(s)ds.
\end{equation}
Notice that multiplying the equation \( B \frac{dY}{dx}+ Q(x)Y=0  \) by \(B^T\), we get a system of the form \eqref{eqn:X0} with \(A=B^TQ\). Recall that \( U(0,x) \) is the fundamental matrix solution to the homogeneous system \( B \frac{dY}{dx}+ Q(x)Y=0  \), see eq. \eqref{eqn:transp}. The following proposition is a direct consequence of the previous results for linear systems of ODEs.
\begin{Prop}\label{prop:nh-dirac-s}
	The matrix solution to the non-homogeneous Dirac equation \( B \frac{dY}{dx}+Q(x)Y=H(x), H\in C([a, b], \mathcal{M}_2)\) with the initial condition \( Y(0)=0 \) is given by
	\begin{equation}
		Y=\mathcal S[H]
	\end{equation}
	where \( \mathcal S \) is the integral operator
	\begin{equation}
		\mathcal S[H](x):=U(0,x)\int_{0}^{x}U^{-1}(0, t)B^T H(t)dt.
	\end{equation}
\end{Prop}
\begin{Rmk}\label{rmk:Uinv}
	After computing \(U(0,x)\), we can obtain \( U^{-1}(0,x)\) without additional calculations. Indeed, due to the well-known formula for the inverse of a \(2\times 2\) matrix we have that
	\[
	U(0,x)=\begin{pmatrix}
		a(x) & b(x) \\ c(x)& d(x)
	\end{pmatrix}, \quad U^{-1}(0,x) = \frac{1}{\det U(0,x)} \begin{pmatrix}
		d(x) & -b(x) \\ -c(x)& a(x)
	\end{pmatrix}.
	\]
	Moreover, the determinant of the fundamental matrix solution \(X_0(t)\) of \eqref{eqn:X0} satisfies the following formula (see \cite[Lemma 1.5]{Hale})
	\[
	\det X_0(t)=\left[\det X_0(a)  \right] \exp\left(-\int_{a}^t \mathrm{tr} A(s)ds\right)=\exp\left(-\int_{a}^t \mathrm{tr} A(s)ds\right).
	\] 
	Since \( \mathrm{tr} \left[BQ(x)\right] =0\), \( \det U(0,x)=1\). Thus,
	\[
	U^{-1}(0,x) =  \begin{pmatrix}
		d(x) & -b(x) \\ -c(x)& a(x)
	\end{pmatrix}.
	\]
\end{Rmk}
We can also compute the value of the operator applied to the product of a scalar function with a matrix of the form \( \mathcal{A}_Q[H] \).
\begin{Lemma}\label{lemma:nonhsol}
	Let \( h \in C^1([0,b]) \) and \( H \in C^1([0,b], \mathcal{M}_2) \). Then,
	\begin{equation}
		\mathcal S[h(\cdot)\mathcal{A}_Q [H]](x)=h(x)H(x)-h(0)U(0,x)H(0)-\mathcal S[h'(\cdot)B H]
	\end{equation}
	In particular, if either \( h(0)=0 \) or \( H(0)=0 \) we have
	\begin{equation}
		\mathcal S[h(\cdot)\mathcal{A}_Q[H]](x)=h(x)H(x)-\mathcal S[h'(\cdot)B H]
	\end{equation}
\end{Lemma}
\begin{proof}
	First, since \( h \) is a scalar function, we have that
	\[
	\begin{split}
		U^{-1}(0,x)(B^T\left(h(x)\mathcal{A}_Q[H] \right)&= h(x)U^{-1}(0,x)B^T \left[ \left( BH'(x)+Q(x)H(x) \right) \right] \\
		&= h(x)\left[  \left( U^{-1}(0,x)H'(x)+U^{-1}(0,x)B^TQ(x)H(x) \right) \right].
	\end{split}
	\]
	It can be shown by means of the  adjoint problem (see eq. \eqref{eqn:adjoint} and eq. \eqref{eqn:X0inv}) that
	\[
	(U^{-1}(0,x))'=U^{-1}(0,x)B^TQ(x).
	\]
	Therefore,
	\[
	\begin{split}
		U^{-1}(0,x)(B^T\left(h(x)\mathcal{A}_Q[H] \right)&=h(x)\left[  \left( U^{-1}(0,x)H'(x)+(U^{-1}(0,x))'H(x) \right) \right] \\
		&=h(x) \left[ U^{-1}(0,x)H(x)\right]'.
	\end{split}
	\]
	Then, integration by parts yields
	\begin{align*}
		\mathcal S[h(\cdot)\mathcal{A}_Q[H]](x)&=U(0,x)\int_{0}^{x}U^{-1}(0,t)(B^T\left(h(t)\mathcal{A}_Q[H](t) \right)dt\\
		&=U(0,x)\int_{0}^{x}h(t) \left[ U^{-1}(0,t)H(t)\right]'dt\\
		&=h(x)H(x)-h(0)U(0,x)H(0)-\mathcal S[h'(\cdot)BH](x).
	\end{align*}
\end{proof}
We now state the main result of this work, which follows directly from proposition \ref{prop:nh-dirac-s} together with lemma \ref{lemma:nonhsol}.
\begin{Thm}\label{thm:thetan}
	Let \( Q \in C^1([0,b],\mathcal{M}_2)\). Then, the solution \(U(\lambda,x)\) of the Dirac equation \eqref{eqn:dirac} with initial condition \(U(\lambda,0)=I\) admits a representation in terms of a Neumann series of Bessel functions (NSBF) given by eq. \eqref{eqn:Unsbf}. Moreover, consider the matrix-valued functions \( \theta_n(x):=x^nK_n(x) \) where \(K_n\) are the coefficients from \eqref{eqn:Unsbf}. Then, the functions \(\theta_n(x)\) can be obtained by the following recursive procedure. The first functions are given by 
	\begin{align*}
		\theta_0(x)&=K_0(x)=\frac{1}{2} \left(U(0,x)-I\right), \\
		\theta_{-1}(x)&=-\frac{1}{x}K_0(x)=-\frac{1}{2x} \left(U(0,x)-I\right),
	\end{align*}
	and for \(n \geq 1\)
	\begin{equation}\label{eqn:thetan}
		\theta_n(x)=\frac{2n+1}{2n-3} \left[ x^2\theta_{n-2}+\mathcal S\left[ -(2n-1)xB\theta_{n-2}+(2n-3)\theta_{n-1}B \right] \right].
	\end{equation}
\end{Thm}
\begin{Rmk}
	One can verify that the last theorem remains true for \( Q \in L^2([0,b],\mathcal{M}_2)\) by taking a sequence \( \{Q_n\}\in C^1([0,b],\mathcal{M}_2)\) that converges to \(Q\). 
\end{Rmk}
\subsection{The NSBF representation of the solution of the ZS-AKNS system}
We consider the \(2\times 2\) system
\begin{equation}\label{eqn:zs-akns}
	\frac{d V}{dx}=Q_{ZS}V+i\lambda \sigma_3 V, \quad V(\lambda,x)=\begin{pmatrix}
		v_1(\lambda, x) \\ v_2(\lambda, x)
	\end{pmatrix},
\end{equation}
where
\begin{equation}
	Q_{ZS}(x)=\begin{pmatrix}
		0 & \nu(x) \\ \overline{\nu(x)} & 0 
	\end{pmatrix}, \quad \sigma_3=\begin{pmatrix}
		1 & 0 \\ 0 & -1
	\end{pmatrix}, \quad \nu \in L^2([0,b],\mathbb{C}),
\end{equation}
and \(\overline{\nu}\) denotes complex conjugation. System \eqref{eqn:zs-akns} is known as the ZS-AKNS system. It was first studied by Zakharov and Shabat \cite{ZS1,ZS2} and Ablowitz, Kaup, Newell, and Segur \cite{AKNS}. Such system appears in the inverse scattering method for solving nonlinear Schrödinger equations. Recently, inverse scattering for ZS-AKNS systems on the half-line was studied in \cite{HrynivManko}. \\

Define \( p_\nu := \mathrm{Im }\,\nu, q_\nu:=- \mathrm{Re }\, \nu \), and let \(U_\nu(\lambda,x)\) be the matrix solution to \eqref{eqn:dirac} with initial condition \(U_\nu(\lambda,0)=I\) for the matrix potential
\[
Q_\nu=\begin{pmatrix}
	p_\nu & q_\nu \\ q_\nu & -p_\nu 
\end{pmatrix}.
\]
Then,
\begin{equation}
	Z(\lambda,x):=A^{-1}U_\nu(\lambda,x)A, \quad A:=\begin{pmatrix}
		i & -i \\ 1 & 1
	\end{pmatrix},
\end{equation}
is the solution to eq. \eqref{eqn:zs-akns} that satisfies the initial condition \(Z(\lambda,0)=I\), see \cite{HrynivManko}. Applying conjugation by A to the NSBF representation for \(U_\nu(\lambda,x)\) (whose coefficients we denote by \(K_{\nu,n}(x)\)) shows that \(Z(\lambda,x)\) admits the NSBF representation
\begin{equation}\label{eqn:Znsbf}
	\begin{split}
		Z(\lambda,x)&=A^{-1}U_0(\lambda,x)A+ 2\sum_{n=0}^\infty (-1)^n A^{-1}K_{\nu,n}(x)Aj_{2n}(\lambda x)\\
		&\, -2\sum_{n=0}^\infty(-1)^n A^{-1}K_{\nu,n}(x)BAj_{2n+1}(\lambda x).
	\end{split}
\end{equation}
The NBSF representation for \(Z(\lambda,x)\) can be used to solve the inverse scattering problem on the half-line for compactly supported potentials. Representation \eqref{eqn:Znsbf} generalizes the NSBF representation recently published in \cite{Lady} since \eqref{eqn:Znsbf} is valid for complex-valued potentials.

\section{Approximation results}\label{sec:approx}
This section is devoted to establish a variety of bounds regarding the convergence of the NSBF representation. We begin the section recalling the relation between the smoothness of the potential \(Q \) with the smoothness of the integral kernel \(K \). The following proposition establishes such relation.
\begin{Prop}\label{prop:smoothness}
	Let \( \Omega^+ = \{ (x,t): 0\leq x \leq b, \left\vert t \right\vert \leq x \} \). Then,
	\begin{enumerate}
		\item If \( Q \in C^r([0,b],\mathcal{M}_2) \) for some \(r \in {\mathbb{N}}_0. \) Then the integral kernel \( K \) satisfies
		\[
		K \in C^r(\Omega^+, \mathcal{M}_2).
		\]
		\item If \( Q \in W^{r,2}([0,b],\mathcal{M}_2) \) for some \(r \in {\mathbb{N}}_0. \) Then the integral kernel \( K \) satisfies
		\[
		K \in W^{r,2}(\Omega^+, \mathcal{M}_2).
		\]
	\end{enumerate}
\end{Prop}
\begin{proof}
	Part (a) is given in \cite[Theorem A.2, Proposition B.1]{Nelson-analytic}. The proof of part (b) is similar.
\end{proof}
We first need an auxiliary result regarding the spectral norm of \( U_0(\lambda, t) \). A direct computation shows that the following holds true.
\begin{Prop}\label{prop:U0norm}
	Let \( U_0(\lambda, t) \) be given by eq. \eqref{eqn:u0}. Then,
	\begin{equation}
		\left\vert U_0(\lambda,t) \right\vert = e^{\left\vert \operatorname{Im} \lambda \right\vert \left\vert t \right\vert}.
	\end{equation}
\end{Prop}
\begin{Rmk}\label{rmk:norms}
	The expansion of the transmutation kernel in terms of the Legendre polynomials in eq. \eqref{eqn:legendre} is entry-wise, that is, we are actually expanding each of the entries \( \kappa_{ij}(x,t) \) in \eqref{eqn:kappan}. However, since we are working with \( \mathcal{M}_2 \)-valued functions, it is natural to give bounds in terms of the operator norm \( \left\vert \cdot \right\vert \) induced by the Euclidean norm in \( \mathbb{C}^2 \). Given a \( d\times d \) matrix \( A \in \mathcal{M}_d \), it is well-known the following equivalence of matrix norms (see \cite[Section 5.6, Problem 23]{matrix})
	\begin{equation}
		d \max_{1\leq i,j \leq d} \left\vert a_{ij} \right\vert \leq d \left\vert A \right\vert \leq d^2 \max_{1\leq i,j \leq d} \left\vert a_{ij} \right\vert.
	\end{equation}
	Thus, adapting known bounds for complex-valued functions to bounds for matrix-valued functions is straightforward.
\end{Rmk}
Now we give a pair of estimates for the pointwise rate of convergence of the truncated series which are analogous to the results \cite[Thm. 3.3, Thm. 4.1]{nsbf}. We denote the partial sum of the series \eqref{eqn:legendre} by
\begin{equation}
	K^N(x,t):=\sum_{n=0}^{N}\frac{1}{x}K_n(x)P_n\left( \frac{t}{x} \right).
\end{equation}
\begin{Thm}
	Let us assume that \( Q \in C^{p+1}([0,b], \mathcal{M}_2) \). Set
	\[
	M:= \max_{0 \leq x \leq b, \vert t \vert \leq x} \left\vert \partial_t^{p+1} K(x,t) \right\vert.
	\]
	Then, for \( N>p+1\)
	\begin{equation}\label{eqn:Knabs-estimate}
		\left\vert K(x,t)-K^N(x,t) \right\vert \leq \epsilon_N(x), 
	\end{equation}
	where
	\begin{equation}
		\epsilon_N(x):=\frac{C_{p,\mathcal{M}_2} M x^{p+1}}{N^{p+1/2}},
	\end{equation}
	and the constant \( C_{p,\mathcal{M}_2}  \) does not depend on neither \( Q \) or \( N \).
\end{Thm}
\begin{proof}
	The proof is analogous to the proof of \cite[Thm. 3.3]{nsbf} with a straightforward modification to adapt the argument to \( \mathcal{M}_2 \)-valued functions, see remark \ref{rmk:norms}.
\end{proof}

Rewrite the series in eq. \eqref{eqn:Unsbf} as
\begin{equation}\label{eqn:Unsbftilde}
	U(\lambda,x)=U_0(\lambda,x)+\sum_{n=0}^{\infty}\tilde{K}_n(x)j_n(\lambda x),
\end{equation}
where
\begin{equation}
	\tilde{K}_n(x):=
	\begin{cases}
		2(-1)^{n/2} K_n(x) & n \text{ even}, \\
		2(-1)^{(n+1)/2}K_n(x)B & n \text { odd}.
	\end{cases}
\end{equation}
Denote the \(N\)-th partial sum of the series \eqref{eqn:Unsbftilde} by \( U^N(\lambda,x) \), that is
\begin{equation}\label{eqn:Uapprox}
	U^N(\lambda, x):=U_0(\lambda, x) + \sum_{n=0}^{N} \tilde{K}_n(x) j_n(\lambda x).
\end{equation}
The following result shows a remarkable property of the partial sums of the NSBF representation, which is the uniform approximation of the truncated series with respect to the spectral parameter for real values of \( \lambda \).
\begin{Thm}\label{thm:UnCp} Assume \( Q\in C^{p+1}([0,b],\mathcal{M}_2)\).
	\leavevmode
	\begin{enumerate}[(a)]
		\item Let \( \lambda \in \mathbb{R} \). Then,
		\begin{equation}
			\left\vert U(\lambda,x)-U^N(\lambda, x) \right\vert \leq 2 x \epsilon_N(x).
		\end{equation}
		\item For \( \lambda \in \mathbb{C} \) with \( \operatorname{Im} \lambda \neq 0 \),
	\end{enumerate}
	\begin{equation}
		\left\vert U(\lambda,x)-U^N(\lambda, x) \right\vert \leq 2 \epsilon_N(x) \frac{e^{\left\vert \operatorname{Im} \lambda \right\vert x}-1}{\left\vert \operatorname{Im} \lambda \right\vert}.
	\end{equation}
\end{Thm}
\begin{proof}
	Due to proposition \ref{prop:U0norm}, if \( \operatorname{Im} \lambda \neq 0 \) we have that
	\begin{align*}
		\left\vert U(\lambda, x)-U^N(\lambda, x) \right\vert &\leq \int_{-x}^{x} \left\vert K(x,t)-K^N(x,t) \right\vert \left\vert U_0(\lambda,t) \right\vert dt \\
		&\leq \epsilon_N(x) \int_{-x}^{x} \left\vert U_0(\lambda, t) \right\vert dt \\
		&\leq 2 \epsilon_N(x) \int_{0}^{x} e^{\left\vert \operatorname{Im} \lambda \right\vert t} dt = 2 \epsilon_N(x)\frac{e^{\left\vert \operatorname{Im} \lambda \right\vert x}-1}{\left\vert \operatorname{Im} \lambda \right\vert}.
	\end{align*}
	This shows that (b) holds true. If \( \lambda \in \mathbb{R} \),  the proof of part (a) is analogous to the proof of part (b) using that \( \left\vert U_0(\lambda,t) \right\vert =1 \).
\end{proof}
We can improve the rate of convergence of the NSBF series if we utilize the \(L^2\) norm in the estimate \eqref{eqn:Knabs-estimate}. We can even weaken the condition \( Q \in C^{p+1}([0,b], \mathcal{M}_2) \), and instead work with \(Q\) in the Sobolev space \( W^{p+1,2}([0,b], \mathcal{M}_2) \). The following theorem presents an approximation result of the Fourier-Legendre series for \( f \in W^{p+1,2}([-1,1], \mathbb{C}) \), which follows directly from \cite[Theorem 1]{Ky92}.
\begin{Thm}
	Let \( f \in W^{p+1,2}([-1,1], \mathbb{C}) \), \(p \geq 0\). Denote \( \Delta(x):=\sqrt{1-x^2} \). Additionally, define the remainder of the best polynomial approximation of \( f\) in \(L^2[-1,1]\) as 
	\[E_{2,N}(f):=\inf_{p\in \Pi_N} \Vert f - p \Vert_{L_2(-1,1)},\]
	where \( \Pi_N\) is the set of all algebraic polynomials of degree at most \(N\). Then 
	\begin{equation}
		E_{2,N}(f)\leq  \frac{C_1(p)}{N^{p+1}} \Vert \Delta^{p+1} f^{(p+1)} \Vert_{L_2(-1,1)},
	\end{equation}
	and the constant \( C_1(p)\) is independent of \( f, N \).
\end{Thm}
Naturally, if we denote \( S^N_f(x) = \sum_{n=0}^{N}a_n {P}_n(x)  \) the \(N \)-th partial sum of the Fourier-Legendre series of a function \( f\in W^{2,p+1}([-1,1], \mathbb{C}) \) we have that 
\begin{equation}\label{eqn:E2N}
	E_{2,N}(f)=	\Vert f - S^N_f  \Vert_2 \leq \frac{C_1(p)}{N^{p+1}} \Vert \Delta^{p+1} f^{(p+1)} \Vert_{L_2(-1,1)}.
\end{equation}
\begin{Prop}\label{prop:uboundK}
	Let \(Q \in W^{p+1,2}([0,b],\mathcal{M}_2)\), \(p\geq 0\). Then, for all \(0\leq m,l \leq p+1\) such that \( m+l\leq p+1\) we have that
	\begin{equation}
		\Vert \partial^l_t\partial^m_x K(x, \cdot) \Vert_{L_2(-x,x)} \leq C(b,Q,Q',\ldots,Q^{(l+m)})
	\end{equation}
\end{Prop}
\begin{proof}
	Consider the change of variables \(\xi = \frac{1}{2}(x+t), \, \eta=\frac{1}{2}(x-t)\). Let \(H(\xi,\eta)=K(x,t).\) According to proposition \ref{prop:smoothness}, \( K\in W^{p+1,2}(\Omega^+,\mathcal{M}_2)\). Hence, \(H(\xi,\eta) \in W^{p+1,2}(\Sigma^+,\mathcal{M}_2)\subseteq C^{p}(\Sigma^+,\mathcal{M}_2) \) where \(\Sigma^+ := \{ (\xi, \eta) \vert 0\leq \xi \leq b, 0\leq \eta\leq b-\xi\} \).	Set \(G_Q(\theta):=BQ'(\theta)+Q^2(\theta)\). Note that \(G_Q\in W^{p,2}([0,b],\mathcal{M}_2)\). The proof of \cite[Theorem A.2]{Nelson-analytic} shows that
	\[
	H_\xi(\xi,\eta)=\frac{1}{2}G_Q(\theta)+\int_{0}^{\eta}G_Q(\xi+v)H(\xi,v)dv,
	\]
	and
	\[
	H_\eta(\xi,\eta)=\int_{0}^{\xi}G_Q(u+\eta)H(u,\eta)du.
	\]
	Therefore, if \(r\leq p\) we have that
	\begin{align*}
		\partial_\xi^r H_\xi(\xi,\eta)&=\frac{1}{2}G_Q^{(r)}(\xi)+\int_{0}^{\eta}\partial_\xi^r \left(G_Q(\xi+v)H(\xi,v) \right)dv \\
		&=\frac{1}{2}G_Q^{(r)}(\xi)+\sum_{k=0}^{r} \binom{r}{k}\int_{0}^{\eta}G_Q^{(r-k)}(\xi+v)\partial_\xi^k H(\xi,v)dv.
	\end{align*}
	Notice that for \(k\leq r\leq p\) a derivate \( \partial_\xi^k H(\xi,v)\) is uniformly bounded since \(H(\xi,\eta) \in C^{p}(\Sigma^+,\mathcal{M}_2) \). Using the Cauchy-Schwarz inequality, we see that
	\begin{align*}
		\vert \partial_\xi^r H_\xi(\xi,\eta) \vert&\leq \frac{1}{2}\vert G_Q^{(r)}(\xi) \vert+\sum_{k=0}^{r} \binom{r}{k}\int_{0}^{\eta} \vert G_Q^{(r-k)}(\xi+v)\vert \vert \partial_\xi^k H(\xi,v)\vert dv\\
		&\leq \frac{1}{2}\vert G_Q^{(r)}(\xi) \vert+\sum_{k=0}^{r} \binom{r}{k}\sqrt{b}\Vert G_Q^{(r-k)} \Vert_{L_2(0,b)}  \Vert \partial_\xi^k H(\xi,\cdot) \Vert_{C(\Sigma^+,\mathcal{M}_2)}  .
	\end{align*}
	Similarly,
	\begin{align*}
		\vert \partial_\eta^r H_\eta(\xi,\eta) \vert&\leq \sum_{k=0}^{r} \binom{r}{k}\int_{0}^{\xi} \vert G_Q^{(r-k)}(u+\eta)\vert \vert \partial_\eta^k H(u,\eta)\vert du\\
		&\leq\sum_{k=0}^{r} \binom{r}{k}\sqrt{b}\Vert G_Q^{(r-k)} \Vert_{L_2(0,b)} \Vert \partial_\eta^k H(\cdot,\eta) \Vert_{C(\Sigma^+,\mathcal{M}_2)} .
	\end{align*}
	For the mixed partial derivatives we have \(H_{\eta \xi}(\xi,\eta)=G_Q(\xi+\eta)H(\xi,\eta)\). The rest of the details are straightforward. 
\end{proof}
\begin{Prop}\label{prop:L2normK-KN}
	Let \(Q \in W^{p+1,2}([0,b])\). Then for all \(N > p+1, 0<x\leq b\)
	\begin{equation}
		\Vert K(x,\cdot)-K^N(x,\cdot) \Vert_{L_2(-x,x)} \leq \epsilon_{2,N}(x)
	\end{equation}
	where 
	\begin{equation}
		\begin{split}
			\epsilon_{2,N}(x)&:=\frac{C_1(p,\mathcal{M}_2)x^{p+1}}{N^{p+1}} 
			\Vert \partial_2^{p+1}K(x,\cdot) \Vert_{L_2(-x,x)}\\
			&\leq \frac{C_2(p,b,Q)x^{p+1}}{N^{p+1}}.
		\end{split}
	\end{equation}
\end{Prop}
\begin{proof}
	For each \( x \in (0,b] \), let \( F(z):=K(x,xz), \, z\in [-1,1]\). Then \( F \in W^{p+1,2}([-1,1],\mathcal{M}_2)\) and \( F^{(p+1)}(z)=x^{p+1} \partial_2^{p+1} K(x,xz)\), where \(\partial_2\) denotes the partial derivative with respect to the second variable. Now, note that \( F^N(z):=\sum_{n=0}^{N}\frac{1}{x}K_n(x)P_n(z)\) is a partial sum of the Fourier-Legendre series of \(F(z)\). By a change of variables, we see that
	\[
	\Vert K(x,\cdot)-K^N(x,\cdot) \Vert_{L_2(-x,x)} = \sqrt{x} \Vert F-F^N \Vert_{L_2(-1,1)}
	\]
	Hence, by \eqref{eqn:E2N}
	\begin{align*}
		\Vert K(x,\cdot)-K^N(x,\cdot) \Vert_{L_2(-x,x)} &=\sqrt{x} \Vert F-F^N \Vert_{L_2(-1,1)} \\
		&\leq \sqrt{x} \frac{C_1(p,\mathcal{M}_2)}{N^{p+1}} \Vert \Delta^{p+1}  F^{(p+1)} \Vert_{L_2(-1,1)} \\
		&= \frac{C_1(p,\mathcal{M}_2)x^{p+3/2}}{N^{p+1}}
		\Vert \Delta^{p+1}\partial_2^{p+1} K(x,x\cdot)\Vert_{L_2(-1,1)}.
	\end{align*}
	On the other hand, \( \vert \Delta^{p+1}(t) \vert \leq 1 \) for \( t\in [-1,1]\) and the change of variables \(t=xz\) gives us
	\begin{align*}
		\Vert \partial_2^{p+1} K(x,x\cdot)\Vert_{L_2(-1,1)}^2 	&=\int_{-1}^{1}\vert \partial_2^{p+1} K(x,xz) \vert ^2 dz\\
		&=\frac{1}{x}\int_{-x}^{x}\vert \partial_2^{p+1} K(x,t) \vert^2 dt \\
		&=\frac{1}{x}	\Vert  \partial_t^{p+1}K(x,\cdot) \Vert_{L_2(-x,x)}^2,
	\end{align*}
	which can be uniformly bounded by a constant that depends on \(p, b, Q\) and its derivates by proposition \ref{prop:uboundK}, which we write as \(C_2(p,b,Q)\) for short. 
\end{proof}
As a corollary, we obtain the decay rate of the coefficients \( K_n(x)\).
\begin{Cor}
	Let \(Q \in W^{p+1,2}([0,b])\). There exist a constant \(C_3(p,b,Q)\) such that
\begin{equation}
	\vert K_n(x) \vert \leq \frac{C_3(p,b,Q) x^{p+3/2} }{n^{p+1/2}}, \quad n>p+2.
\end{equation}
\end{Cor}
\begin{proof}
	From eq. \eqref{eqn:Kncoeff} and the orthogonality of the Legendre polynomials,
	\begin{align*}
		K_n(x) &= \frac{2n+1}{2} \int_{-x}^{x} K(x,t)P_n(\frac{t}{x}) dt\\
		&= \frac{2n+1}{2} \int_{-x}^{x} \left(K(x,t) - K^{n-1}(x,t) \right)P_n\left(\frac{t}{x}\right) dt.
	\end{align*}
	The Cauchy-Schwarz inequality together with proposition \ref{prop:L2normK-KN} imply that 
	\begin{align*}
		\vert K_n(x) \vert &\leq \Vert K(x,\cdot)-K^{n-1}(x,\cdot)\Vert_{L_2(-x,x)} \Vert P_n \left(\frac{\cdot}{x}\right) \Vert_{L_2(-x,x)} \\
		&\leq \sqrt{\frac{2n+1}{2}}\frac{C_2(p,b,Q)x^{p+3/2}}{(n-1)^{p+1}}\\
		& \leq \frac{C_3(p,b,Q)x^{p+3/2}  }{n^{p+1/2}}.
	\end{align*}
\end{proof}
Now, we present an analogue result to theorem \ref{thm:UnCp}, for a potential \(Q\) in the Sobolev space \( W^{p+1,2}([0,b],\mathcal{M}_2)\). 
\begin{Thm}
	Let \( Q \in W^{p+1,2}([0,b],\mathcal{M}_2)\).
	\begin{enumerate}[(a)]
		\item For \( \lambda \in \mathbb{R} \),
		\begin{equation}
			\left\vert U(\lambda,x)-U^N(\lambda, x) \right\vert \leq \sqrt{2 x} \epsilon_{2,N}(x).
		\end{equation}
		\item If \( \lambda \in \mathbb{C} \) with \( \operatorname{Im} \lambda \neq 0 \),
	\end{enumerate}
	\begin{equation}
		\left\vert U(\lambda,x)-U^N(\lambda, x) \right\vert \leq  \epsilon_{2,N}(x)\sqrt{\frac{e^{2\left\vert \operatorname{Im} \lambda \right\vert x}-1}{\left\vert \operatorname{Im} \lambda \right\vert}}.
	\end{equation}
\end{Thm}
\begin{proof}
	If \( \operatorname{Im} \lambda \neq 0 \), using the Cauchy-Schwarz inequality together with proposition \ref{prop:L2normK-KN}  we see that
	\begin{align*}
		\left\vert U(\lambda, x)-U^N(\lambda, x) \right\vert &\leq \int_{-x}^{x} \left\vert K(x,t)-K^N(x,t) \right\vert \left\vert U_0(\lambda,t) \right\vert dt \\
		&\leq \epsilon_{2,N}(x) \left( \int_{-x}^{x} \left\vert U_0(\lambda, t) \right\vert^2 dt \right)^{1/2} \\
		&= \sqrt{2} \epsilon_{2,N}(x) \left( \int_{0}^{x} e^{2\left\vert \operatorname{Im} \lambda \right\vert t} dt \right)^{1/2}=  \epsilon_{2,N}(x)\sqrt{\frac{e^{2\left\vert \operatorname{Im} \lambda \right\vert x}-1}{\left\vert \operatorname{Im} \lambda \right\vert}}.
	\end{align*}
	This shows that (b) holds true. If \( \lambda \in \mathbb{R} \),  the proof of part (a) is analogous to the proof of part (b) using that \( \left\vert U_0(\lambda,t) \right\vert =1 \).
\end{proof}

Finally, we can prove that the series \eqref{eqn:Unsbf} can be differentiated term by term. 

\begin{Prop}
	Suppose that \( Q \in W^{1,2}([0,b],\mathcal{M}_2)\). Then the series \eqref{eqn:Unsbf} can be termwise differentiated. 
\end{Prop}
\begin{proof}
	Proceeding as before, we rewrite the series \eqref{eqn:Unsbf} as in \eqref{eqn:Unsbftilde}
	\( \sum_{n=0}^\infty \tilde{K}_n(x)j_{n}(\lambda x). \)
	For any \(\lambda\), using the Cauchy-Schwarz inequality, we can see that the series \eqref{eqn:Unsbftilde}
	is absolutely bounded,
	\begin{align*}
		\sum_{n=0}^\infty \vert \tilde{K}_n(x)\vert \vert j_{n}(\lambda x) \vert &\leq\left(\sum_{n=0}^\infty  \frac{2\vert \tilde{K}_n(x) \vert^2}{2n+1} \right)^{1/2} \left(\sum_{n=0}^\infty \frac{(2n+1)\vert j_n(\lambda x)\vert^2 }{2}  \right)^{1/2}\\
		&=\left(2\sum_{n=0}^\infty  \frac{2\vert {K}_n(x) \vert^2}{2n+1} \right)^{1/2} \left(\sum_{n=0}^\infty \frac{(2n+1)\vert j_n(\lambda x)\vert^2 }{2}  \right)^{1/2}.
	\end{align*} 
	Since \( K_n(x) \) are the Fourier-Legendre coefficients of the function \( x K(x,xz) \in L^2([-1,1],\mathcal{M}_2)\) (with respect to \(z\)), using Bessel's inequality and the fact that \( \Vert K(x,\cdot) \Vert_{L_2(-x,x)}\) is uniformly bounded (see proposition \ref{prop:uboundK}) for \( 0<x\leq b\)  we have that
	\begin{align*}
		\sum_{n=0}^\infty \vert \tilde{K}_n(x)\vert \vert j_{n}(\lambda x) \vert &\leq 2 \Vert x K(x,x\cdot) \Vert_{L_2(-1,1)} \left(\sum_{n=0}^\infty \frac{(2n+1)\vert j_n(\lambda x)\vert^2 }{2}  \right)^{1/2}\\
		&\leq C(b,Q ) \left(\sum_{n=0}^\infty \frac{(2n+1)\vert j_n(\lambda x)\vert^2 }{2}  \right)^{1/2}.
	\end{align*} 
	We recall the following bound for the spherical Bessel functions \cite[f. 9.1.62]{Abra}
	\[
	\vert j_n(z) \vert \leq \sqrt{\pi} \left\vert \frac{z}{2}\right\vert^n \frac{e^{\vert \operatorname{Im}z \vert}}{\Gamma(n+3/2)}.
	\]
	Hence,
	\begin{align*}
		\sum_{n=0}^\infty \vert \tilde{K}_n(x)\vert \vert j_{n}(\lambda x) \vert &\leq C(b,Q) \left(\sum_{n=0}^\infty \frac{(2n+1)}{2}\pi \left \vert \frac{\lambda x}{2} \right \vert^{2n} \frac{ e^{2x \vert \operatorname{Im}\lambda \vert} }{\Gamma^2(n+3/2)}   \right)^{1/2}\\
		&\leq C(b,Q ) \left(\pi\sum_{n=0}^\infty  \left \vert \frac{\lambda b}{2} \right \vert^{2n} \frac{ e^{2b \vert \operatorname{Im}\lambda \vert} }{\Gamma(n+3/2)\Gamma(n+1/2)}   \right)^{1/2}.
	\end{align*}
	Since the series in the right-hand side is convergent, it follows that \eqref{eqn:Unsbf} is uniformly convergent. The coefficients \( K_n\in W^{1,2}(0,b)  \) which can be seen for instance, using the mapping property (see theorem \ref{thm:mapping} ). Moreover, for the derivatives we have
	\[
	K'_n(x)=\frac{2n+1}{2}\int_{-1}^1(K(x,xz)+xK_1(x,xz)+xzK_2(x,xz))P_n(z)dz.
	\]
	This shows that \( K'_n(x)\) are the Fourier Legendre coefficients of the function \(K(x,xz)+xK_1(x,xz)+xzK_2(x,xz) \in L^2[-1,1] \). Here \(K_1\) and \(K_2\) stand for the partial derivatives with respect to the first and second variable. By proposition \ref{prop:uboundK}
	\[
	\Vert K_1(x,\cdot) \Vert_{L_2(-x,x)}, \Vert K_2(x,\cdot) \Vert_{L_2(-x,x)} \leq  C(b,Q,Q' ).
	\]
	Then, an analogous argument shows that the series of derivatives given by eq. \eqref{eqn:dUnsbf} is uniformly convergent. Hence, it can be termwise differentiated. 
\end{proof}

\section{Numerical implementation and examples}\label{sec:numerical}
In this section we describe an algorithmic procedure to solve either an initial value problem or a boundary value problem for the one-dimensional Dirac equation \eqref{eqn:dirac}, based on the results of the previous sections. Our method relies on a truncation of the NSBF series \eqref{eqn:Unsbf}. A natural question arises immediately. Is there a practical way to control the accuracy of a partial sum approximation? We answer this question in the following observation.
\begin{Rmk}\label{rmk:accuracy}
	Substitution of the Fourier-Legendre series of the transmutation kernel \(K(x,t)\) in the Goursat conditions \eqref{eqn:goursatx} and \eqref{eqn:goursatmx} leads to the following equalities 
	\begin{equation}
		\frac{1}{x} \sum_{n=0}^{\infty} \left[BK_n(x)-K_n(x)B  \right]  = -Q(x), \quad \frac{1}{x} \sum_{n=0}^{\infty}(-1)^n \left[BK_n(x)-K_n(x)B  \right] = 0,
	\end{equation}
	where we have used the fact that the \(n\)-th Legendre polynomial satisifes \( P_n(1)=1, \, P_n(-1)=(-1)^n.\) Hence, the following differences
	\begin{equation}
		\begin{split}
					\delta_{Q,N}(x)&:= \bigg \vert Q(x) + \frac{1}{x} \sum_{n=0}^{N} \left[BK_n(x)-K_n(x)B  \right]  \bigg\vert, \\ 
					\delta_{0,N}&:= \bigg \vert\frac{1}{x} \sum_{n=0}^{N} (-1)^n \left[BK_n(x)-K_n(x)B  \right]  \bigg\vert,
		\end{split}
	\end{equation}
	indicate the accuracy of the approximation of the transmutation kernel. We can use this fact to control the accuracy of the approximate solution \eqref{eqn:Uapprox} and estimate an optimal number \(N\) to choose.
\end{Rmk}
Let us consider the one-dimensional Dirac equation \eqref{eqn:dirac} with an initial condition
\begin{equation}\label{eqn:ic}
	Y(0)=\begin{pmatrix}{y_{1}(0)}\\{y_{2}(0)}\end{pmatrix}=\begin{pmatrix}{c_1}\\{c_2}\end{pmatrix},
\end{equation}
or a boundary condition
\begin{equation}\label{eqn:bc}
	\begin{pmatrix}{a_{11}}&{a_{12}}\\{a_{21}}&{a_{22}}\end{pmatrix}\begin{pmatrix}{y_{1}(0)}\\{y_{2}(0)}\end{pmatrix}+\begin{pmatrix}{a_{13}}&{a_{14}}\\{a_{23}}&{a_{24}}\end{pmatrix}\begin{pmatrix}{y_{1}(b)}\\{y_{2}(b)}\end{pmatrix}=\begin{pmatrix}{0}\\{0}\end{pmatrix}.
\end{equation}
In principle, the coefficients \(a_{ij}\) in the boundary condition can be not only constants but possibly, entire functions of the spectral parameter \(\lambda\). The algorithm to solve either of these problems is as follows.
\begin{enumerate}
	\item Find the matrix fundamental solution $U(0,x)$ of the equation $B\frac{dY}{dx}+Q(x)Y=0$. Such solution can be found using any numerical method to solve initial value problems for ODEs. In particular, it can be constructed using the SPPS method, see \cite[Theorem 2.1 \& Section 2.3]{Nelson-spps}. 
	\item Compute the matrix functions $K_{n}$, $n=0,\ldots,N$, using eq. \eqref{eqn:thetan}, see also theorem \ref{thm:thetan} and remark \ref{rmk:Uinv}. To choose an optimal number \(N\) to achieve a desired accuracy, see remark \ref{rmk:accuracy}.
	\item Compute the approximate solution \( U^N(\lambda, x) \), see eq. \eqref{eqn:Uapprox}.
	\item An approximation to the solution of the initial value problem \eqref{eqn:ic} is given by $Y^{N}(\lambda, x)=U^N(\lambda, x)\begin{pmatrix}c_1 \\ c_2 \end{pmatrix}$.
	\item The characteristic function of the spectral problem defined by the boundary condition \eqref{eqn:bc} is given by
		\begin{equation}\label{eqn:charfun}
	\Delta(\lambda):=\operatorname{det}\left(M(\lambda)\right),
	\end{equation}
	where
	\begin{equation}\label{eqn:Mchareq}
		M(\lambda)=
		\begin{pmatrix}
			{a_{11}}&{a_{12}}\\
			{a_{21}}&{a_{22}}
		\end{pmatrix}
		+
		\begin{pmatrix}
			{a_{13}}&{a_{14}}\\
			{a_{23}}&{a_{24}}
		\end{pmatrix}
		U(\lambda,b).
	\end{equation}
	The eigenvalues coincide with the zeros of the characteristic function \eqref{eqn:charfun}. We find the zeros of an approximation of the characteristic function \(	\Delta_N(\lambda) =\operatorname{det}\left(M_N(\lambda)\right)\), with \(M_N(\lambda)\) given by
\begin{equation}
		M_N(\lambda)=
		\begin{pmatrix}
			{a_{11}}&{a_{12}}\\
			{a_{21}}&{a_{22}}
		\end{pmatrix}
		+
		\begin{pmatrix}
			{a_{13}}&{a_{14}}\\
			{a_{23}}&{a_{24}}
		\end{pmatrix}
		U^N(\lambda,b).
\end{equation}
	
\end{enumerate}
All the computations were performed numerically using Matlab 2024b in machine precision. Further implementation details can be found in \cite{Nelson-spps, nsbf} and the references therein. The zeros of the characteristic function were found using the Matlab function \texttt{fzero}.

\begin{Rmk}
	For numerical methods that benefit from or depend on the knowledge of the \(\lambda\)-derivative of the matrix solution, we can compute it directly as
	\begin{equation}
			\begin{split}
				\frac{d}{d \lambda}U(\lambda, x)  
				& = U_{0,\lambda}(\lambda, x) + 2 \sum_{n=0}^{\infty}(-1)^n K_{2n}(x)\left[\frac{2n}{\lambda}j_{2n}(\lambda x)-xj_{2n+1}(\lambda x) \right]\\
				&\,\,-2\sum_{n=0}^{\infty}(-1)^n K_{2n+1}(x)B\left[  xj_{2n}(\lambda x) - \frac{2n+2}{\lambda}j_{2n+1}(\lambda x)\right],
			\end{split}
	\end{equation}
	where 
	\begin{equation}
		U_{0,\lambda}(\lambda, x) = \begin{pmatrix}
			-x \sin \lambda x & -x \cos \lambda x \\ x \cos \lambda x & -x \sin \lambda x
		\end{pmatrix}.
	\end{equation}
\end{Rmk}

\begin{Ex}[Example 3.4 \cite{AnnabyTharwat2}]\label{ex:AT}
	
	Consider the following spectral problem
	\begin{equation}\label{eqn:example}
		\begin{pmatrix}
			{0}&{1}\\
			{-1}&{0}
		\end{pmatrix}
		\frac{dZ}{dx}
		+\begin{pmatrix}
			{-x}&{0}\\
			{0}&{1}
		\end{pmatrix}
		Z=\lambda Z,
		\quad
		Z=\begin{pmatrix}
			{u(x)}\\
			{v(x)}
		\end{pmatrix},
		\quad
		0\le x\le 1,
	\end{equation}
	with boundary conditions
	\begin{equation}\label{eqn:bcexample}
		u(0)=u(1)=0.
	\end{equation}
	We can transform \eqref{eqn:example} into the form \eqref{eqn:dirac} using the orthogonal transformation
	\begin{equation*}
		Z(x)=\begin{pmatrix}
			{\cos(\varphi(x))}&{-\sin(\varphi(x))}\\
			{\sin(\varphi(x))}&{\cos(\varphi(x))}
		\end{pmatrix}
		Y(x),\quad \varphi(x)=\frac{x(x-2)}{4},
	\end{equation*}
	see \cite{levitan}. For more details, see \cite{Nelson-analytic}. After transforming the system and the boundary conditions appropriately, we compute 240 eigenvalues, numbered from -105 to 134, using \( N=16. \) In figure \ref{fig:ejemplo} we show the absolute error of the computed eigenvalues. Note that for large eigenvalues, the calculated ones coincide with the exact eigenvalues up to machine precision. 
	\begin{figure}
		\centering
		\includegraphics[width=5in,height=2.4in]{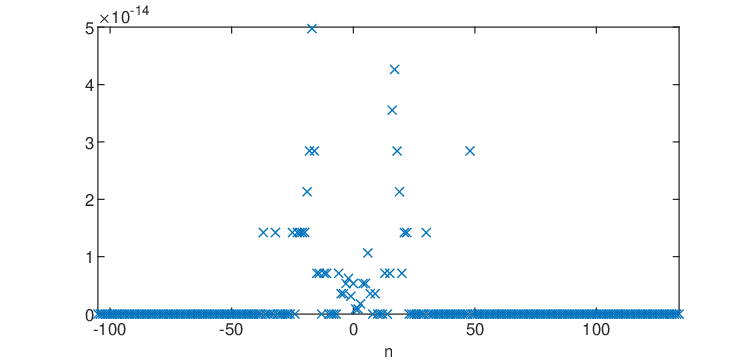}
		\label{fig:ejemplo}
		\caption{Absolute error of the eigenvalues \( \lambda_n \) of example \ref{ex:AT}, \( N=16\).}
	\end{figure}
\end{Ex}

\appendix

\section{The mapping property and its connection to the NSBF representation}\label{appendix:mapping}

Let \( (f,g)^T \) be a solution of the homogeneous Dirac equation \( B \frac{dY}{dx}+Q(x)Y=0 \), such that both \(f,g \) do not vanish in the interval \( [0,b]  \). Assume that the solution \( (f,g)^T \) is normalized by the condition \( f(0)g(0)=1 \).
Consider the following systems of functions defined by the recursive relations \cite{Nelson-spps},
\begin{align}
	X^{(0)}(x)&=-\int_{0}^{x}\frac{p(s)}{f^(s)}ds, \quad Y^{(0)}(x)=1+\int_{0}^{x}\frac{p(s)}{g^2(s)}ds, \label{eqn:XY0} \\
	Z^{(n)}(x)&=\int_{0}^{x}\left(f^2(s)X^{(n)}(s)+g^2(s)Y^{(n)}(s) \right)ds, \label{eqn:Zn} \\
	X^{(n+1)}(x)&=-(n+1)\int_{0}^{x}\left(\frac{g(s)}{f(s)}Y^{(n)}(s)+\frac{p(s)}{f^2(s)}Z^{(n)}(s) \right)ds, \label{eqn:Xn}  \\
	Y^{(n+1)}(x)&=(n+1)\int_{0}^{x}\left(\frac{f(s)}{g(s)}X^{(n)}(s)+\frac{p(s)}{g^2(s)}Z^{(n)}(s) \right)ds, \quad n=0,1,2,\ldots  \label{eqn:Yn}
\end{align}
Similarly, let us set
\begin{equation}
	\tilde{X}^{(0)}(x)=1+\int_{0}^{x}\frac{p(s)}{f^(s)}ds, \quad \tilde{Y}^{(0)}(x)=-\int_{0}^{x}\frac{p(s)}{g^2(s)}ds, \label{eqn:XYtilde0}
\end{equation}
and define \(\tilde{X}^{(n+1)}, \tilde{Y}^{(n+1)}\) and \(\tilde{Z}^{(n)}\) by replacing \(X^{(n)}, 
Y^{(n)}\) and \(Z^{(n)} \)  for \(\tilde{X}^{(n)}, \tilde{Y}^{(n)}\) and \(\tilde{Z}^{(n)}\) respectively in eqs. \eqref{eqn:Zn}, \eqref{eqn:Xn} and \eqref{eqn:Yn} .
Now, consider the infinite sequence of vector-valued functions \( \{ \Phi_k \}_{k=0}^\infty \) and \( \{ \Psi_k \}_{k=0}^\infty \) given by \cite{Nelson-analytic}
\begin{equation}\label{eqn:phik}
	\Phi_k = \begin{cases}
		(-1)^{(k-1)/2} f(0) \begin{pmatrix} fX^{(k)} \\ gY^{(k)} \end{pmatrix} & k \text{ odd}, \\
		(-1)^{k/2} g(0) \begin{pmatrix} f\tilde{X}^{(k)} \\ g\tilde{Y}^{(k)} \end{pmatrix} & k \text{ even}, \\
	\end{cases}
\end{equation}
and
\begin{equation}\label{eqn:psik}
	\Psi_k = \begin{cases}
		(-1)^{(k-1)/2}g(0) \begin{pmatrix} f\tilde{X}^{(k)} \\ g\tilde{Y}^{(k)} \end{pmatrix} & k \text{ odd}, \\
		(-1)^{k/2} f(0)\begin{pmatrix} f{X}^{(k)} \\ g{Y}^{(k)} \end{pmatrix} & k \text{ even}. \\
	\end{cases}
\end{equation}
Then the following result known as the mapping property holds.
\begin{Thm}\cite[Thm 3.3]{Nelson-analytic}\label{thm:mapping}
	Let \( p,q \in L^1(0,b) \) be complex valued functions. Let \( T \) be the transmutation operator for \(A_0  \) and \( A_Q \), and let \( \Phi_k \) and \( \Psi_k \) be vector-valued functions defined by \eqref{eqn:phik} and \eqref{eqn:psik} respectively. Then
	\begin{equation}
		T \begin{pmatrix} x^k \\ 0 \end{pmatrix} = \Phi_k, \quad
		T \begin{pmatrix} 0 \\ x^k \end{pmatrix} = \Psi_k, \quad k=0,1,2,\ldots
	\end{equation}
\end{Thm}
Let \(P_n \) denote the Legendre polynomial of order \(n \) and \( l_{k,n} \) be the corresponding coefficient of \( x^k \), that is \(P_n(x)=\sum_{k=0}^{n} l_{k,n}x^k \).
\begin{Thm}\label{thm:coeffmap}
	The coefficients of the Fourier-Legendre series \eqref{eqn:legendre} of the transmutation kernel \(K(x,t) \)
	are given by
	\begin{equation}
		K_n(x)=\frac{2n+1}{2}\left[ -I + \sum_{k=0}^{n}\frac{l_{k,n}}{x^k}\eta_k(x)\right]
	\end{equation}
	where
	\begin{equation}
		\eta_k(x)=\left( \Phi_k(x) \,\, \Psi_k(x) \right),
	\end{equation}
	is a \( 2 \times 2 \)-matrix whose columns are \( \Phi_k \) and \( \Psi_k \).
\end{Thm}
\begin{proof}
	The proof is almost the same as the proof of \cite[Thm 3.3]{nsbf}.
	Using \eqref{eqn:Kncoeff}, we have that
	\begin{align*}
		K_n(x) &= \frac{2n+1}{2} \int_{-x}^{x} K(x,t)P_n \left( \frac{t}{x} \right)dt= \frac{2n+1}{2} \sum_{k=0}^{n} \int_{-x}^{x} K(x,t) l_{k,n}
		\left( \frac{t}{x} \right)^k \cdot I dt \\
		&= \frac{2n+1}{2} \sum_{k=0}^{n} \frac{ l_{k,n} }{x^k} \int_{-x}^{x} K(x,t) t^k \cdot I dt= \frac{2n+1}{2} \sum_{k=0}^{n} \frac{l_{k,n}}{x^k} \left(T[x^k\cdot I] -x^k\cdot I \right).
	\end{align*}
	Using Theorem \ref{thm:mapping} and the fact that \( \sum_{k=0}^n l_{k,n}=P_n(1)=1 \) we obtain the desired result.
\end{proof}

\bibliographystyle{plain}

\end{document}